\documentclass[12pt]{amsart}
\usepackage{amsmath,amssymb,amsthm}
\usepackage{ marvosym }
\usepackage{tikz,tikz-cd}
\usepackage{framed}
\usepackage{lscape}
\usepackage{wasysym}
\usepackage[utf8]{inputenc}
\usepackage{mathtools}
\usepackage[makeroom]{cancel}
\usepackage{afterpage}
\usepackage{color}
\usepackage{enumerate}
\usepackage{amsfonts}
\usepackage{dsfont}
\usepackage{mathrsfs}
\newtheorem{definition}{Definition}
\newtheorem{theorem}[definition]{Theorem}
\newtheorem{lemma}[definition]{Lemma}
\newtheorem{proposition}[definition]{Proposition}

\newtheorem{example}[definition]{Example}
\usepackage{geometry} 
\allowdisplaybreaks
\numberwithin{definition}{section}

\title{$c_2$ Invariants of Recursive Families of Graphs}
\author{Wesley Chorney and Karen Yeats}
 
\begin{document}

\thanks{Wesley Chorney was supported by an NSERC USRA.  Karen Yeats is supported by an NSERC Discovery grant.}

\begin{abstract}
The $c_2$ invariant, defined by Schnetz in \cite{Schnetz2011}, is an arithmetic graph invariant created towards a better understanding of Feynman integrals. \\
This paper looks at some graph families of interest, with a focus on decompleted toroidal grids. Specifically, the $c_2$ invariant for $p=2$ is shown to be zero for all decompleted non-skew toroidal grids.  We also calculate $c_2^{(2)}(G)$ for $G$ a family of graphs called X-ladders. Finally, we show these methods can be applied to any graph with a recursive structure, for any fixed $p$.
\end{abstract} 
\maketitle
\section{Introduction}
Given a connected, 4-regular graph $\Gamma$, let $G=\Gamma\backslash v$, where $v\in V(\Gamma)$. We call $G$ a decompletion of $\Gamma$ and write $G=\widetilde{\Gamma}$. In this way, $G$ can be thought of as a Feynman graph in $\phi^4$ theory with four external edges. Note that in general, this is bad notation since the decompletion of a graph is non-unique. However, for the graphs appearing in this document, all decompletions but one (see $\S \ref{sec xladder}$) are isomorphic and so the decompletion operation is well-defined. \\
\begin{definition}Assign to each edge $e\in G$ a variable $\alpha_e$. The \textbf{Kirchhoff polynomial} of $G$ is 
	\[\Psi_G=\sum_{T}\prod_{e\not\in T}\alpha_e,\]
	where the sum is over all spanning trees in $G$.
\end{definition}
We use the Kirchhoff polynomial to define the Feynman period of $G$ as 
\[\int_{\alpha_i\geq0}\frac{\Omega}{\Psi_G^2}\]
where $\Omega=\sum_i^{|E(G)|}(-1)^{i-1}d\alpha_1\cdots \widehat{d\alpha_i}\cdots d\alpha_{|E(G)|}$, and $\widehat{d\alpha_i}$ corresponds to the differential not appearing in the product. The Feynman period is interesting both quantum field theoretically and mathematically.  From the point of view of quantum field theory it is an important part of the complete Feynman integral (see \cite{Schnetz2010}).  More mathematically, the Feynman period is the right kind of object to try to understand with algebro-geometric tools.  There has been substantial work over the last decade taking this approach, see \cite{BlochEsnaultKreimer2006, BroadhurstSchnetz2014, Brown2009, Marcolli2010}. Schnetz \cite{Schnetz2011} defined the $c_2$ invariant, given below, in order to better understand these integrals.\\
\begin{definition}Let $p$ be a prime, $\mathds{F}_p$ the finite field with $p$ elements, and let $[\Psi_G]_p$ denote the cardinality of the affine algebraic variety of $\Psi_G$ over $\mathds{F}_p$. Further, suppose $G$ has at least 3 vertices. Then the \textbf{$c_2$ invariant of $G$ at $p$} is
	\[c_2^{(p)}(G)=\frac{[\Psi_G]_p}{p^2}\bmod p\]
\end{definition}
The $c_2$ invariant is well-defined provided $G$ has at least three vertices \cite{Schnetz2011}. It is or is predicted to be invariant under the symmetries of the Feynman period \cite{BrownSchnetz2012, Doryn2013}; knowing the $c_2$ invariant provides important information about the Feynman period. \\ \\
The graphs of main interest in this document are toroidal grids (specifically 2-dimensional ones), which are interesting not only from a quantum field theory perspective, but also graph theoretically. For instance, as shown in \cite{DeGraafShrijver1994}, any graph with face width $r\geq5$ embedded on a torus contains a certain toroidal grid as a minor. Furthermore, their maximal run length \cite{Doig2004} and bent Hamilton cycle properties \cite{RuskeySawada2003} have been investigated.
\begin{definition}
  A \textbf{toroidal grid} is a graph defined in the following way.  Given two integer vectors $(k,0)$, $(l,m)$ with $k, m \geq 3$ and $l\geq 0$, take the integer lattice points in the first quadrant with edges joining lattice points at distance $1$.  The result of this modulo the relation which identifies two lattice points if their difference is $(k, 0)$ or $(l, m)$ is the toroidal grid indexed by $(k, 0)$ and $(l, m)$.
\end{definition}

A toroidal grid with $l \neq 0$ is called a \textbf{skew} toroidal grid while those with $l = 0$ are \textbf{non-skew}.

\begin{proposition}
	The toroidal grid indexed by $(k, 0)$ and $(0, m)$ is a Cartesian product of cycles $\gamma_k \times \gamma_m$, where $\gamma_i$ is the cycle on $i$ vertices.  
\end{proposition}

\begin{proof}
This follows directly from the definition.
\end{proof}


\begin{example}
	Let $\mathbf{x}=(3,0)$ and $\mathbf{y}=(0,3)$. Figure \ref{fig eg} shows the lattice and resulting graph. Notice the graph corresponds exactly to $\gamma_3\times \gamma_3$. \\
\end{example}

	\begin{figure}[h]
	\includegraphics[scale=0.8]{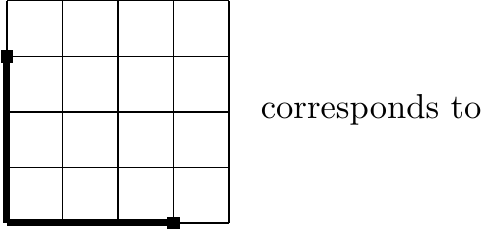}
	\includegraphics[scale=1]{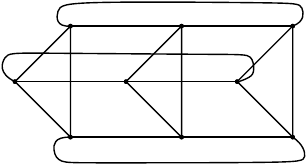} 
	\caption{Example of a toroidal grid}\label{fig eg}
	\end{figure}
	
        The sources above involving toroidal grids do not include skew toroidal grids in the definition. 
        
        With appropriate relative primality of the parameters, toroidal grids can also be understood as certain circulant graphs.
        \begin{definition}
          The \textbf{circulant graph} $C_n(i_1, i_2, \ldots, i_k)$ is the graph on $n$ vertices with an edge between vertices $i$ and $j$ iff $i-j \equiv i_\ell \mod n$ or $j-i \equiv i_\ell \mod n$ for some $\ell$.
        \end{definition}
        
\begin{proposition}
	Let $G$ be a skew toroidal grid, parametrized by $\mathbf{x}=(k,0), \mathbf{y}=(l,m)$ with $l > 0$ and $\gcd(m, l)=1$. Then $G$ is isomorphic to the circulant graph $C_{km}(l,m)$. 
\end{proposition}
\begin{proof}
  Let $G$ be the skew toroidal grid parametrized by $(k,0)$ and $(l,m)$.
  Take the integer lattice points in the first quadrant with $x$ coordinate less than $k$ and $y$ coordinate less than $m$ as representatives for the vertices of $G$.

  Next we will label the vertices of $G$ with $\{1, \ldots, km\}$ so as to indicate the circulant structure.  Let the vertex $(a,b)$ be labelled with $1 + am + (m-b-1)l \mod km$, see Figure~\ref{fig tor lab}.   Every label is used exactly once because row $b$ of the grid uses precisely the labels congruent to $(m-b-1)l+1$ modulo $m$ and since $\gcd(m, l)=1$ this runs over all the equivalence classes as $b$ runs over $0\leq b < m$.

  The horizontal edges of the grid connect $1+am+(m-b-1)l$ with $1+(a+1)m + (m-b-1)l$ for $0 \leq a < k-1$ and $0 \leq b < m$.  Additionally from the horizontal toroidality we have edges connecting $1+(k-1)m+(m-b-1)l$ with $1+(m-b-1)l$ for $0\leq b < m$.  This gives all the gap $m$ edges for the circulant structure.  The vertical edges of the grid connect $1+am+(m-b-1)l$ with $1+am+(m-b)l$ for $0\leq b < m-1$ and $0\leq a < k$.  The remaining toroidality gives edges connecting $1+am$ and $1+(a-l)m+(m-1)l = 1+am -l$ which gives all the gap $l$ edges for the circulant structure.  This accounts for all the edges of $G$.
%
   %
\end{proof}
\begin{figure}[h]
	\includegraphics[scale=1.0]{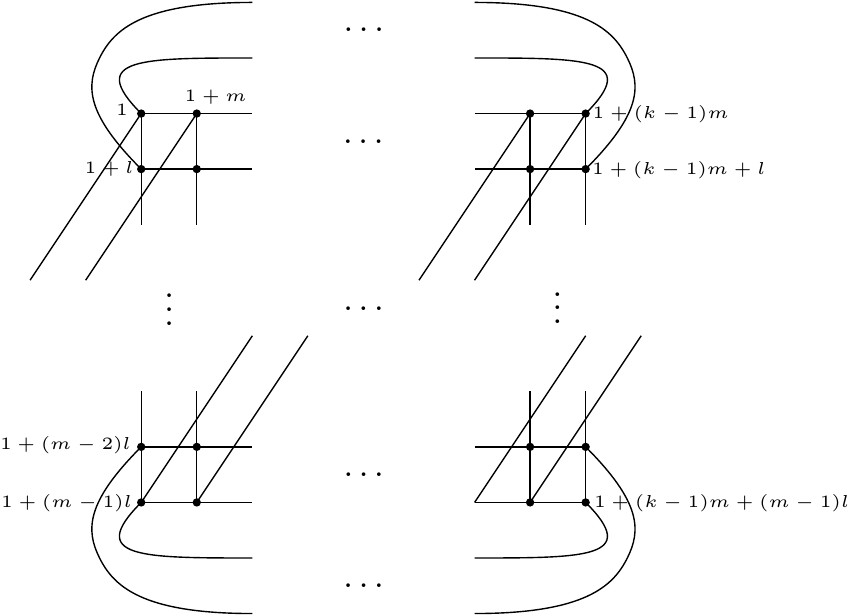}
	\caption{Toroidal Grid Labelling.}\label{fig tor lab}
\end{figure}

These are examples of the kinds of families of circulant graphs whose $c_2$ invariants were studied in \cite{Yeats2016}.  In particular, from \cite{Yeats2016} we know that $c_2^{(2)}(\widetilde{C}_n(1,3)) \equiv n \mod 2$ for $n\geq 7$ so the decompleted toroidal grid parametrized by $(k, 0)$, $(1, 3)$ has $c_2^{(2)} \equiv 3k \equiv k \mod 2$ for all $k > 2$.  

There is a similar result for non-skew toroidal grids.
\begin{proposition}
	Let $G$ be a non-skew toroidal grid, parametrized by $\mathbf{x}=(k,0), \mathbf{y}=(0,m)$ with $\gcd(m, k)=1$. Then $G$ is isomorphic to the circulant graph $C_{km}(k,m)$. 
\end{proposition}

\begin{proof}
Similarly to the previous proposition let $G$ be the toroidal grid parametrized by $\mathbf{x}=(k,0), \mathbf{y}=(0,m)$ and take the integer lattice points in the first quadrant with $x$ coordinate less than $k$ and with $y$ coordinate less than $m$ as representatives for the vertices of $G$.  Again we will label the vertices of $G$ with $\{1, \ldots, km\}$ so as to indicate the circulant structure.

Specifically, label the vertex $(a,b)$ with $1+am+bk \mod km$.  Row $b$ uses the labels congruent to $1+bk \mod m$.  Since $\gcd(k,m)=1$ every label occurs exactly once in the graph and the cycles for each row give the edges linking vertices at distance $m$ in the circulant structure.  The same argument with $k$ and $m$ reversed gives that the column cycles give the edges linking the vertices at distance $k$ in the circulant structure and this accounts for all the edges of $G$.
\end{proof}

Note that in the non-skew case this does not give a family of circulants of the form studied in \cite{Yeats2016} because of how the gap parameters depend on the size.  So for the purposes of the $c_2$ invariant these are new graphs to consider and are the main object of study of this paper.

\medskip

In this document, we first define and give some preliminary results, in order to move from an algebraic incarnation of the $c_2$ invariant towards a graph-theoretic or combinatorial understanding. Then, with these methods, we compute $c_2^{(2)}(G)$ where $G$ is a toroidal grid of arbitrary length constructed from $N$-cycles, for $N\geq 3$. We also use these methods to show $c_2^{(2)}(G)=0$ when $G$ is a capped X-ladder, a result already known but proved easily via these methods. Finally, we show that for any recursive family of graphs and any fixed prime $p$, the $c_2$ invariant can be computed for all graphs of the family by a finite procedure using these methods --- giving the possibility of an (unfortunately inefficient) algorithm. \\

\section{Graph polynomials}
Herein, we define a slew of polynomials which will be useful in moving towards a graph-theoretic understanding of the $c_2$ invariant. By the matrix-tree theorem, we can express $\Psi_G$ as a determinant as follows. Choosing an arbitrary orientation of the edges of $G$, let $E$ be the signed incidence matrix (with rows indexing vertices and columns indexing edges) with one row removed. Let $\Lambda$ be the matrix with the edge variables of $G$ on the diagonal and zeroes elsewhere. Let 
\[M=\begin{bmatrix}
\Lambda & E^T \\
-E & 0
\end{bmatrix}\]
Then 
\[\Psi_G=\det M.\]
The proof can be found in \cite{Brown2009}, where the determinant is expanded, or in \cite{VlasevYeats2012}, using the Schur complement and Cauchy-Binet formula. \\
Let $I$ and $J$ be sets of indices, and $M(I,J)$ the matrix $M$ with rows indexed by elements of $I$ and columns indexed by elements of $J$ removed. Then we can define \textbf{Dodgson polynomials} as did Brown in \cite{Brown2009}.\\
\begin{definition}
	Let $I,J,K$ be subsets of ${1,2,\dots,|E(G)|}$, and let $|I|=|J|$. Then
	\[\Psi_{G,K}^{I,J}=\det M(I,J)|_{a_e=0, e\in K}\]
\end{definition}
If the graph is made clear from the context, we leave out the $G$ subscript. Similarly, if $K$ is empty, we leave it out as well. Note that if $e\in I\cap J$, $e\not\in K$, then both the row and column corresponding to $e$ are removed. This is equivalent to $e$ not being in the graph. Specifically,
\[\Psi_{G,K}^{Ie,Je}=\Psi_{G\backslash e, K}^{I,J}\]
Similarly, if $e\in K$, $e\not\in I\cup J$, then edge $e$ is set to zero, but not removed from the matrix. That is, we are taking only those monomials where $e$ does not appear --- equivalently, those monomials where $e$ is not cut in the spanning structure. Specifically, 
\[\Psi_{G,Ke}^{I,J}=\Psi_{G/e,K}^{I,J}\]
These equivalences simplify some steps in the $c_2$ calculations to follow and should be kept in mind by the reader. \\
Dodgson polynomials can be expressed in terms of spanning forests. The following \textbf{spanning forest polynomials} allow us to do so in a relatively straightforward manner. 
\begin{definition}
	Let $P$ be a set partition of a subset of $V(G)$. Define
	\[\Phi_G^P=\sum_F\prod_{e\not\in F}\alpha_e\]
	where the sum runs over all spanning forests $F$ of $G$ with a bijection between the trees of $F$ and the parts of $P$, and where vertices belonging to a part lie in their corresponding tree.
\end{definition}
Note that trees consisting of a single vertex are allowed. We illustrate vertices belonging to different parts by using differing large vertex shapes. 
\begin{example}
	Figure \ref{fig sp for G} shows a graph $G$ with illustrated partition $P=\{\blacksquare, \ocircle\}$. The resulting spanning forest polynomial is
	\[\Phi_G^P=c(de + ae + bd + ab)+ab(e+d)+de(a+b)\]
\end{example}
\begin{figure}[h]
\includegraphics{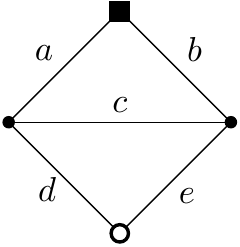}
\caption{$G$ for spanning forest example.}\label{fig sp for G}
\end{figure}
The expression for Dodgson polynomials in terms of spanning forest polynomials is given in \cite{BrownYeats2011} by the following proposition. 
\begin{proposition}\label{prop sp for}
	Let $I$, $J$, $K$ be sets of edge indices of $G$ with $|I|=|J|$. Then
	\[\Psi_{G,K}^{I,J}=\sum_P \pm\Phi_{G\backslash (I\cup J\cup K)}^P\]
	where the sum runs over all set partitions $P$ of the endpoints of the edges in $(I\cup J\cup K)\backslash(I\cap J)$ with the additional property that all forests corresponding to $P$ become trees in both $G\backslash I/(J\cup K)$ and $G\backslash J/(I\cup K)$.
\end{proposition}
\cite{BrownYeats2011} also shows how to determine the sign. However, we will do computations modulo 2 and so sign is irrelevant. This proposition is how spanning forest polynomials typically arise for us. Again, when the graph is clear, we will leave out the subscript. In this case, it is assumed that the graph we are working with is $G$ with all necessary edges left out. \\
It remains to give an expression by which the $c_2$ invariant can be calculated. Once again, Dodgson polynomials are useful for this purpose.
\begin{definition}
	Let $i,j,k,l,m$ be distinct edge indices of $G$. Then the \textbf{5-invariant} of $G$ depending on $i,j,k,l,m$ is
	\[^5\Psi(i,j,k,l,m)=\pm(\Psi_m^{ij,kl}\Psi^{ikm,jlm}-\Psi_m^{ik,jl}\Psi^{ijm,klm})\]
\end{definition}
Up to sign, this is independent of the order of $i,j,k,l,m$; as shown in Lemma 87 of \cite{Brown2009}.\\
Now, following from Lemma 24 and Corollary 28 of \cite{BrownSchnetz2012}, we have these expressions for the $c_2$ invariant:
\begin{proposition}\label{prop calculate c2}
	Suppose $G$ satisfies $2+|E(G)| \leq 2|V(G)|$. Let $i,j,k,l,m$ be distinct edge indices of $G$, and let $p$ be a prime. Then
	\begin{enumerate}
	\item $c_2^{(p)}(G)=-[\Psi_k^{i,j}\Psi^{ik,jk}]_p \bmod p$
	\item $c_2^{(p)}(G)=[\Psi^{ij,kl}\Psi^{ik,jl}]_p \bmod p$
	\item $c_2^{(p)}(G)=-[^5\Psi(i,j,k,l,m)]_p \bmod p$
	\end{enumerate}
	Once again, $[\cdot]_p$ denotes the cardinality of the affine variety over $\mathds{F}_p$. 
\end{proposition}
Note that the graphs of interest in $\phi^4$ theory are decompleted 4-regular graphs. The reader can easily verify that using any expression above and converting it into spanning forest polynomials will yield a graph that satisfies the criterion. \\
We give one last result --- a lemma from the proof of the Chevalley-Warning theorem --- after which a general method for computing the $c_2$ invariant is outlined.
\begin{lemma}\label{lem c2 coeff}
	Let $F$ be a polynomial of degree $N$ in $N$ variables with integer coefficients. Then the coefficient of $x_1^{p-1}x_2^{p-1}\cdots x_N^{p-1}$ in $F^{p-1}$ is $[F]_p$ modulo p. 
\end{lemma}
The proof can be found in section 2 of \cite{Ax1964}. This lemma is of key importance in finding the size of the affine variety modulo $p$ for the $c_2$ invariant above. Given a graph $G$ and one of the expressions in proposition~\ref{prop calculate c2}, working modulo 2 we need only to assign each edge of the graph once between the two polynomials. To simplify this, we convert the polynomials above to spanning forest polynomials, by proposition~\ref{prop sp for}. By itself, this is not so useful since if $p$ and $G$ are both fixed there are many finite ways to compute $c_2^{(p)}(G)$ including simple brute force counting and denominator reduction \cite{BrownSchnetz2012}.  However, using Proposition~\ref{prop calculate c2} and Lemma~\ref{lem c2 coeff} is particularly useful because unlike other techniques they first let us interpret the calculations combinatorially as edge assignments, and more importantly, they allow us to work recursively and obtain finite formulas for entire families of graphs. The explicit calculations which follow will clarify matters. 

\section{Non-skew toroidal grids}
In this section we will show that all decompleted non-skew toroidal grids have $c_2^{(2)} = 0$.  The proof will be done by fixing $m\geq 3$ and considering the family of toroidal grids indexed by $(k, 0)$ and $(0, m)$ for all $k\geq 3$.  To illustrate the argument we will first prove the $(k, 0)$, $(0, 3)$ case separately and then proceed to the general case.

It is interesting that all decompleted non-skew toroidal grids have $c_2^{(2)}=0$ because when the $c_2$ invariant is $0$ is important.  If $c_2^{(p)} =0$ for all $p$ then the graph's Feynman period should have less than the maximal transcendental weight for the size of the graph, see \cite{Schnetz2011}.  This is known as weight drop, see \cite{BrownYeats2011}. The interpretation of $c_2^{(p)}=0$ only for $p=2$ is less clear.   We know some reasons why $c_2^{(p)}$ may be 0 for a graph, see \cite{ BrownSchnetz2012, BrownSchnetzYeats2014}, but none of these apply to non-skew toroidal grids.  New weight drop graphs are likely to be quite sparse, so one should remain pessimistic about how many of the non-skew toroidal grids will turn out to have weight drop.  However, the non-skew toroidal grids still provide a very interesting family with $c_2^{(2)}=0$, whether or not it is for reasons other than weight drop and possibly even giving new families of weight drop graphs.  Calculating the $c_2^{(3)}$ for some of these graphs would be particularly interesting in order to try to distinguish the different possibilities.

\begin{proposition}\label{prop 3tor}
	Let $G$ be a decompleted toroidal grid constructed from 3-cycles, with $|V(G)|\geq8$, and with the edges and vertices of $G$ labelled as in figure \ref{fig G 3}. Then $c_2^{(2)}(G)=0$.
\end{proposition}
\begin{figure}[h]
	\includegraphics[scale=0.75]{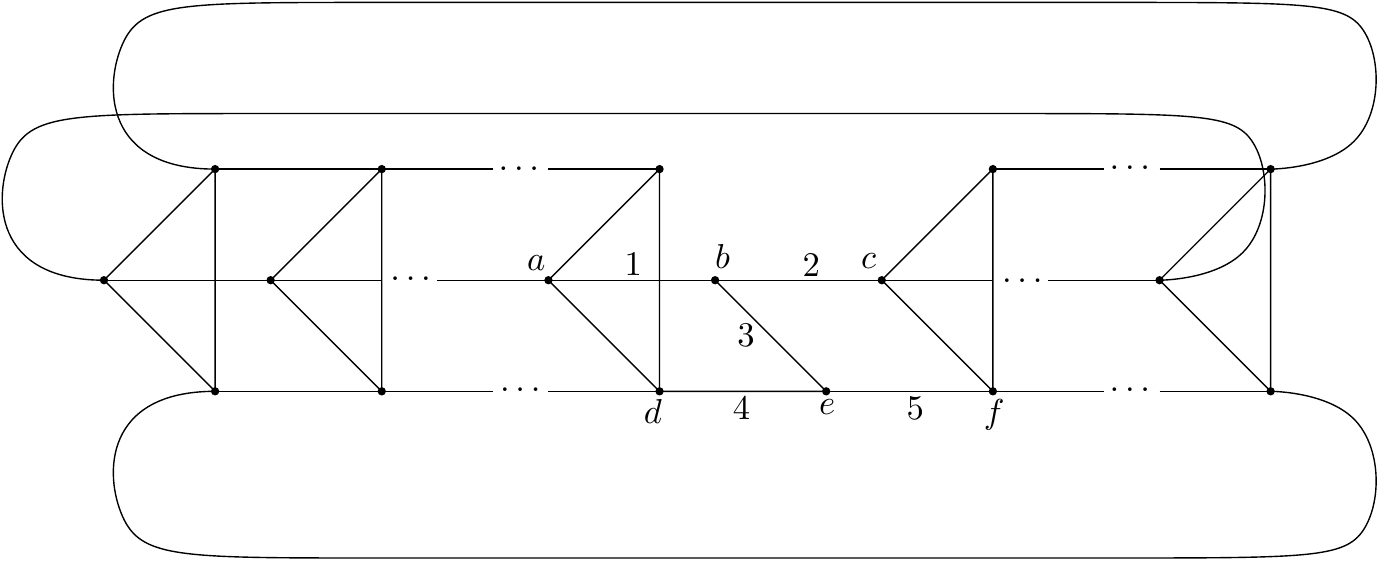}
	\caption{$G$ for 3-toroidal grid.}\label{fig G 3}
\end{figure}
When $|V(G)|=8$, this corresponds to $P_{7,10}$ in \cite{Schnetz2010}. \\ 
\begin{proof}
	Using the 5-invariant, we calculate 
	\[\Psi_3^{12,45}=\pm \Phi^{\{a,d\},\{c,f\},\{b\},\{e\}} \pm \Phi^{\{a,f\},\{c,d\},\{b\},\{e\}}\]
	\[=\includegraphics[scale=0.5]{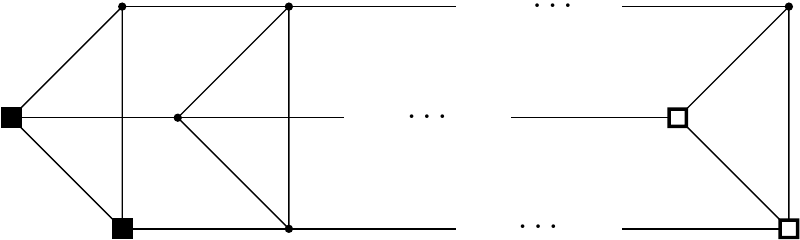} \quad \pm \includegraphics[scale=0.5]{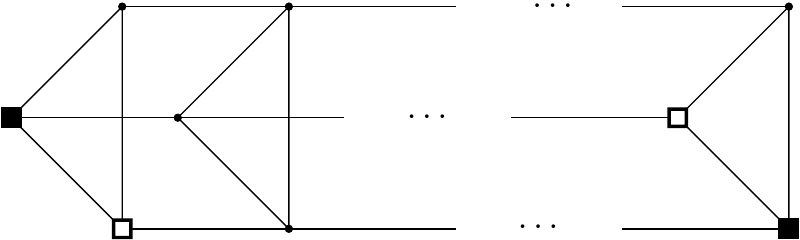}\]
	and
	\[\Psi^{134,235}=   \pm \Phi^{\{a,c,d,f\},\{b\},\{e\}}=\Psi_H\]
	
	\noindent Where $H$ is the graph in figure \ref{fig H 3}, below. We disregard the other term since $\Psi^{123,345}=0$.  \\
	\begin{figure}[h]
		\includegraphics[scale=0.75]{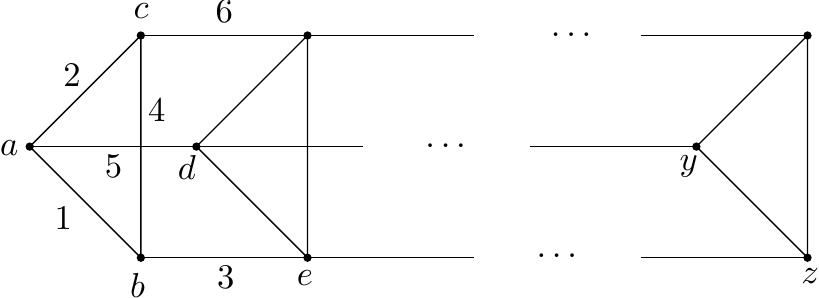} 
		\caption{$H$}\label{fig H 3}
	\end{figure}
	We compute a recurrence below to obtain the coefficient of $\alpha_1\alpha_2\cdots \alpha_{|E(H)|}$ in ${}^5\Psi$. Where necessary, we write $H_n$ to indicate the $n$-dependence of $H$, where $n$ is the number of vertices. Let $a_n=[\Psi_{H_n}(\pm \Phi_{H_n}^{\{a,b\},\{y,z\}} \pm \Phi_{H_n}^{\{a,z\},\{b,y\}})]_2$. By lemma~\ref{lem c2 coeff}, we need to assign each edge to either $\Psi$ or $\Phi$ in  $\Psi_H\Phi_H^{\{a,b\},\{y,z\}}$ and likewise in $\Psi_H\Phi_H^{\{a,z\},\{b,y\}}$ (note that these are the same as the polynomials calculated originally, only relabelled to match $H$). We say an edge is assigned if it appears in the corresponding spanning tree or spanning forest structure. If any assignment of edges is not invariant under the symmetry of $H$ where we swap vertices $a,b;d,e;...;y,z$ pairwise, then the flipped assignment is valid as well and so these cancel modulo 2.
        
	Also, in every case, vertex $a$ cannot be disconnected in $\Psi_H$. Thus one of 1, 2, or 5 must be in $\Psi_H$. In fact, we claim $\Psi_H\Phi_H^{\{a,z\},\{b,y\}}$ yields no contribution modulo 2. If $1\in \Psi_H$, then because of the symmetry, to get a nonzero contribution we must have edges 2,4 in $\Phi_H^{\{a,z\},\{b,y\}}$ --- but this would connect vertices $a$ and $b$. Similarly, if $2\in \Psi_H$, 4 must be as well. Thus 1 must be in $\Phi_H^{\{a,z\},\{b,y\}}$, but this again connects $a$ and $b$. Finally, if $5\in \Psi_H$ and neither of edges 1 or 2, we get a cycle in $\Phi_H^{\{a,z\},\{b,y\}}$. Thus the term itself does not contribute.
        
	We now turn our attention to $\Psi_H\Phi_H^{\{a,b\},\{y,z\}}$. We claim similarly that $\Psi_H\Phi_H^{\{a,b\},\{y,z\}}$ yields no contribution modulo 2.

\medskip
        
	\textbf{Case 1.} $1\in \Psi_H$. \\
	Then $2,4 \in \Phi_H^{\{a,b\},\{y,z\}}$ so that there is no cycle in $\Psi_H$. Furthermore, to avoid disconnecting $a,b$ in $\Psi_H$, we must have $3,5\in \Psi_H$. Similarly, we must have $6\in \Psi_H$ or $c$ will be disconnected in $\Psi_H$.
        
	Therefore, with only $2,4\in \Phi_H^{\{a,b\},\{y,z\}}$, the rest of $H$ must be spanned, so this becomes $\pm \Psi_{H_{n-3}}$. With $1,3,5,6 \in \Psi_H$, to avoid cycles we must not connect $d$ and $e$ in the spanning tree structure, but we must connect one of $a$ or $b$ to $c$. This is $\Phi_{H_{n-3}}^{\{a\},\{b\}}$.  As a whole this case has the same contribution as $\Psi_{H_{n-3}}\Phi^{\{a\}, \{b\}}_{H_{n-3}}$.

\medskip
        
	\textbf{Case 2.} $2,4\in \Psi_H$. \\
	Then $1\in \Phi_H^{\{a,b\},\{y,z\}}$ so that there is no cycle in $\Psi_H$. Now if $3,5\in \Phi_H^{\{a,b\},\{y,z\}}$ as well as edge 6, $a,b,c$ will be disconnected from the rest of the graph in $\Psi_H$. If instead 6 is in $\Psi_H$, then $c$ is disconnected in $\Phi_H^{\{a,b\},\{y,z\}}$, which cannot happen. Thus $3,5\in \Psi_H$. If edge 6 is in $\Psi_H$ as well, then once again $c$ is disconnected. Thus the only permissible assignment of edges is with $1,6\in \Phi_H^{\{a,b\},\{y,z\}}$ --- this becomes $\Psi_{H_{n-3}}$ --- and $2,3,4,5 \in \Psi_H$. In $\Psi_H$, we cannot connect vertices $d$ and $e$. On $H_{n-3}$ labelled as in figure \ref{fig H 3}, this becomes $\Phi_{H_{n-3}}^{\{a\},\{b\}}$, which cancels with case 1. \\ \\
	Therefore, $c_2^{(2)}(G)=0$ for $|V(G)|\geq8$.
\end{proof}

Next, we generalize the arguments of the previous section in order to show that if $G$ is any non-skew toroidal grid, then its $c_2$ invariant at 2 is zero.
\begin{proposition}
	Let $G$ be a decompleted toroidal grid constructed from $N$-cycles, with $|V(G)| \geq  3N-1$. Then $c_2^{(2)}(G)=0$.
\end{proposition}

\begin{proof}
	With $G$ labelled as in figure \ref{fig non skew}, we begin with $\Psi^{12,45}$ and $\Psi^{15,24}$ and proceed to assign edges according to lemma~\ref{lem c2 coeff}. To avoid disconnecting $a$ and $b$, we must assign edges $3,6$ to $\Psi^{12,45}$ and thus the two factors we have are $\Psi_{36}^{12,45}$ and $\Psi^{1536,2436}$. \\
	\begin{figure}[h]
		\includegraphics[scale=0.7]{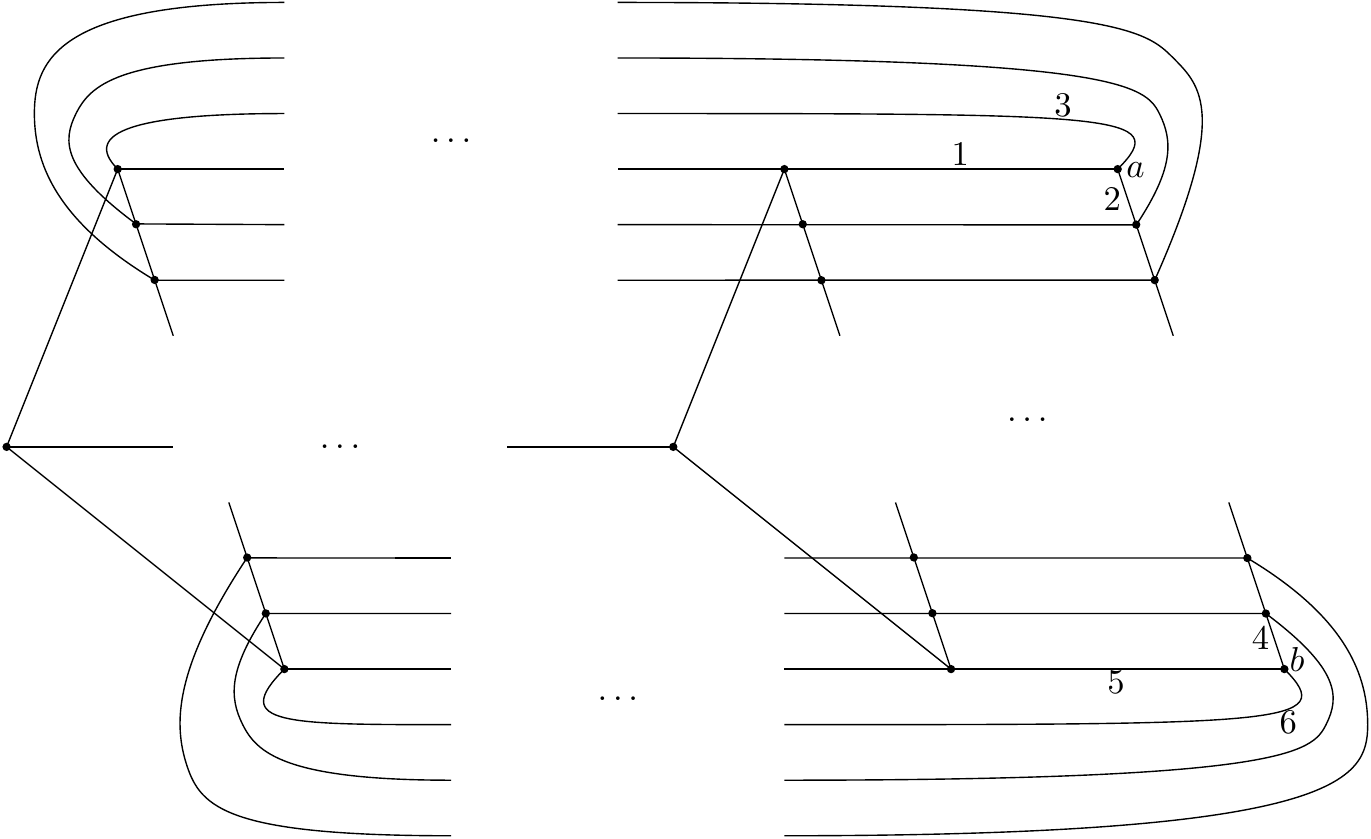}
		\caption{Decompleted non-skew toroidal grid.}\label{fig non skew}
	\end{figure}
	With $H$ as in figure \ref{fig gen H}, $\Psi^{1536,2436}$ is simply $\Psi_H$, and we have the following result for $\Psi_{36}^{12,45}$, with respect to the labelling of $H$:
	\begin{align*}
          \Psi_{36}^{12,45}=& \pm\Phi^{\{a,d\},\{b,e\},\{c,f\}}\pm\Phi^{\{a,d\},\{b,f\},\{c,e\}}\pm\Phi^{\{a,f\},\{b,e\},\{c,d\}} \\
	  & \pm\Phi^{\{a,e\},\{b,d\},\{c,f\}}\pm\Phi^{\{a,f\},\{b,d\},\{c,e\}}\pm\Phi^{\{a,e\},\{b,f\},\{c,d\}}
        \end{align*}
	\begin{figure}[h]
		\includegraphics[scale=0.7]{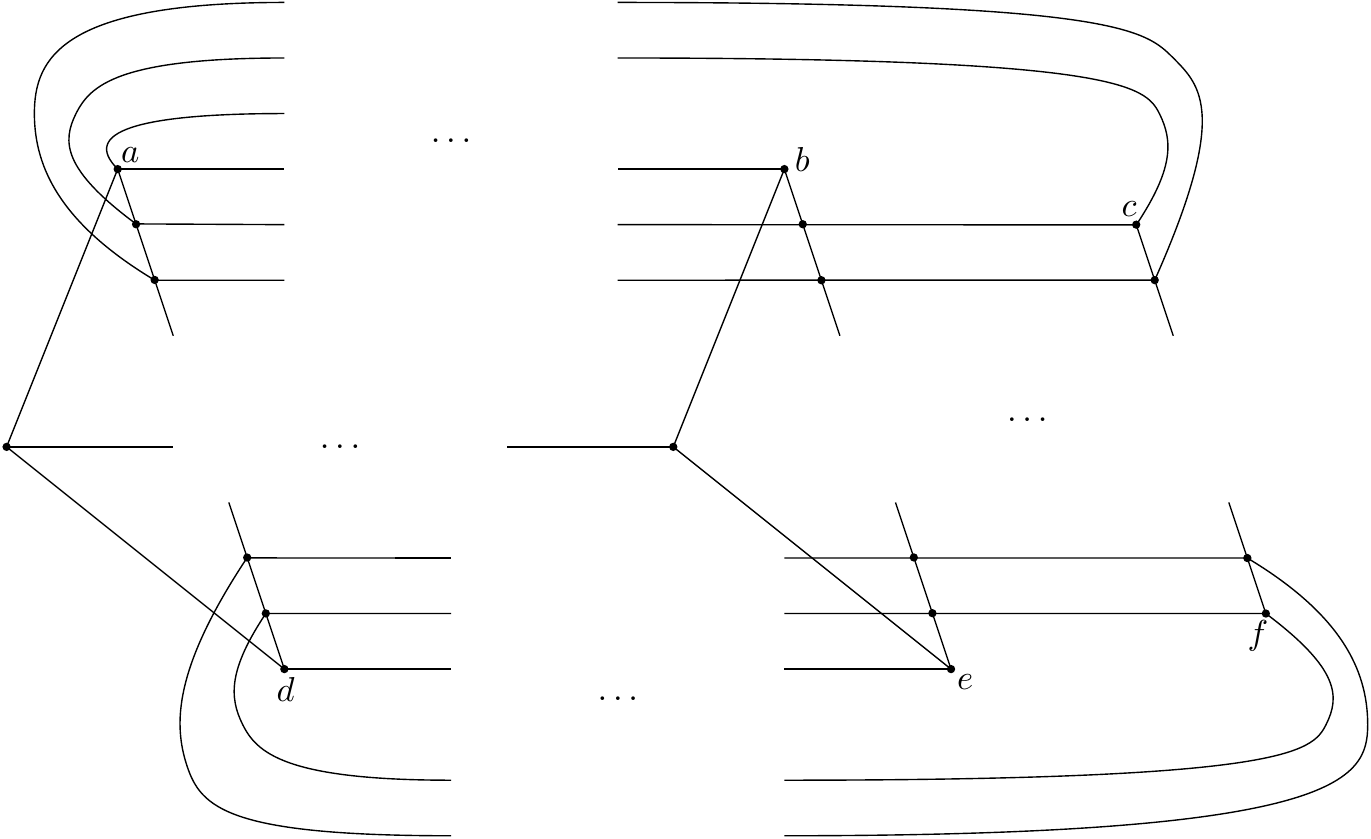}
		\caption{$H$}\label{fig gen H}
	\end{figure}
	Notice that $H$ possesses vertical symmetry (the flip where $a$ goes to $d$, $b$ to $e$, etc.) --- thus the spanning forest polynomials above must be invariant under this vertical flip in order to make a nonzero contribution modulo 2 (as in the proof of the previous proposition). From here, we see $\Phi^{\{a,f\},\{b,d\},\{c,e\}}$ and $\Phi^{\{a,e\},\{b,f\},\{c,d\}}$ do not contribute modulo 2. \\
	Notice further that $H$ possesses horizontal symmetry (the flip where $a$ goes to $b$, $d$ to $e$, etc. while $c$ and $f$ remain fixed). Again, terms must be invariant under this horizontal flip or they will not contribute modulo 2. Thus we need only consider
	\[
        \pm\Phi^{\{a,d\},\{b,e\},\{c,f\}}\pm\Phi^{\{a,e\},\{b,d\},\{c,f\}}\]
        in calculating edge assignments for $\Psi_{36}^{12,45}$.
        
	We consider the path from $c$ to $f$. In particular, we consider the possible assignments for edges $7,8,\dots,12$, as shown in figure \ref{fig spine}.
        
	\begin{figure}[h]
		\includegraphics[scale=0.8]{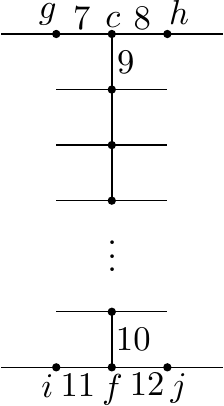} 
		\caption{Labelling around $c$ and $f$}\label{fig spine} 
	\end{figure}
        
	Once again, the edge assignments must be invariant under the horizontal and vertical flip. Furthermore, since $c$ and $f$ cannot be entirely disconnected from $H$ in either factor of $\Psi_H(\Phi^{\{a,d\},\{b,e\},\{c,f\}}\pm\Phi^{\{a,e\},\{b,d\},\{c,f\}})$, the only possibilities are assigning edges $7,8,11,12$ to one factor, and $9,10$ to the other. We proceed via cases.

\medskip
	
	\noindent\textbf{Case 1}. Suppose edges $9,10$ are assigned to $\Phi^{\{a,d\},\{b,e\},\{c,f\}}$ or to $ \Phi^{\{a,e\},\{b,d\},\{c,f\}}$. Consider now the $\Psi_H$ assignment.  Vertices $g$ and $h$ connect in $\Psi_H$ via edges $7$ and $8$ and so $g$ and $h$ cannot be connected in $H\backslash\{7,8,9,10,11,12\}$.  Likewise for $i,j$. To avoid disconnecting the graph, we must connect exactly one of $g$ or $h$ with exactly one of $i$ or $j$. However, none of these possibilities are invariant under both a horizontal and vertical flip, and therefore this case does not contribute. 

\medskip
	
	\noindent\textbf{Case 2}. Suppose edges $7,8,11,12$ are assigned to $\Phi^{\{a,d\},\{b,e\},\{c,f\}}$ or to $ \Phi^{\{a,e\},\{b,d\},\{c,f\}}$.  Now $g$ and $h$ are connected in this factor by edges $7$ and $8$ so, similarly to the previous case, when then considering this half of the edge assignment on $H\backslash\{7,8,9,10,11,12\}$, again we cannot connect $g$ to $h$ or $i$ to $j$, but exactly one of $g$ or $h$ must connect to exactly one of $i$ or $j$, since $c$ and $f$ must be in a tree. However, like the case above, none of these configurations are invariant under both a horizontal and vertical flip, and so this case does not contribute. 

\medskip
	
	Therefore, $c_2^{(2)}(G)=0$ for $|V(G)| \geq 3N-1$. 
\end{proof}

\section{X-ladders}\label{sec xladder}
We call the graphs in figure \ref{fig X ladder} \textit{X-ladders}, either symmetric or capped based on their ends. The $c_2$ invariant of the decompleted capped X-ladders is well known, shown to be zero for all $p$ using double-triangle reduction, see section 5.6 of \cite{PanzerSchnetz2016}.  These graphs are particularly interesting because not only are their $c_2$ invariants $0$, indicating a drop in transcendental weight, but in fact those which have been calculated have maximal multiple weight drop.
Using our methods, we get an easy alternate way to see $c_2^{(2)}(G)=0$ for the capped X-ladders which applies with minor modifications to the symmetric X-ladders. By definition we cannot see higher weight drop in the $c_2$ invariant, but it is suggestive that our arguments essentially only used the symmetries from one X of the ladder leaving the symmetries of the remaining Xs to potentially be somehow showing the further weight drops.\\
\begin{figure}[h]
\includegraphics[scale=0.5]{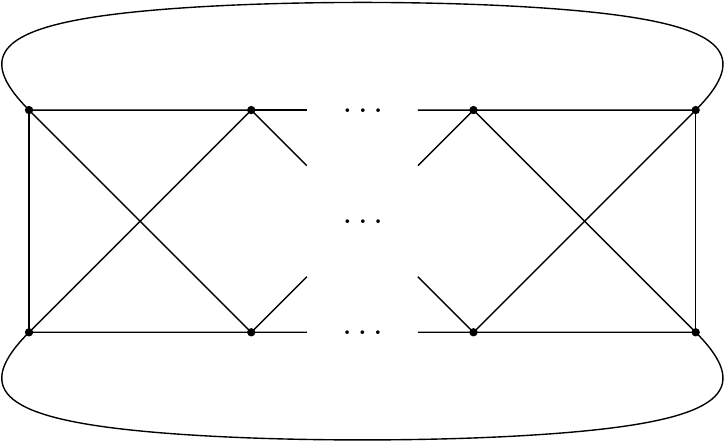} \qquad
\includegraphics[scale=0.5]{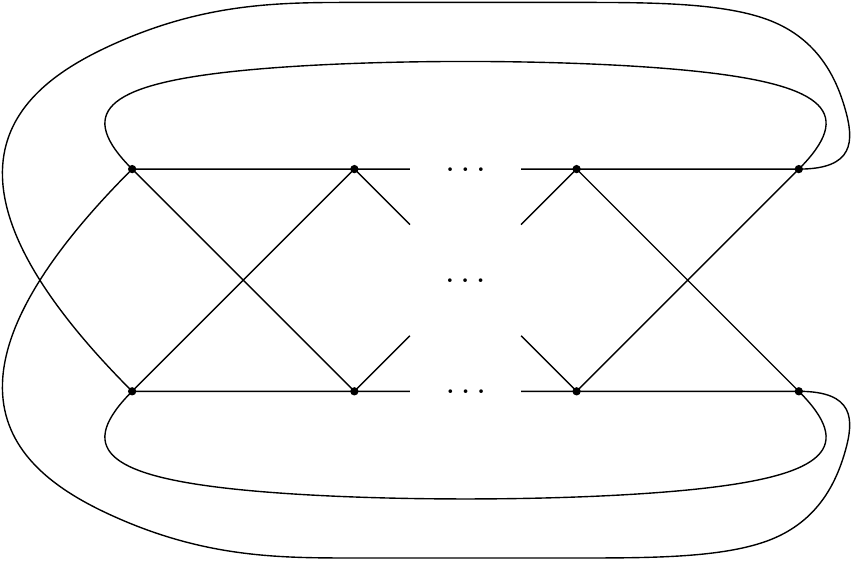}
\caption{Capped (left) and symmetric (right) X-ladders.}\label{fig X ladder}
\end{figure}
\begin{proposition}
	Let $G$ be a decompleted capped X-ladder, with $|V(G)|\geq 7$, labelled as in figure \ref{fig anti X}. Then $c_2^{(2)}(G)=0$. \\
\end{proposition}

\begin{figure}[h]
		\includegraphics[scale=0.75]{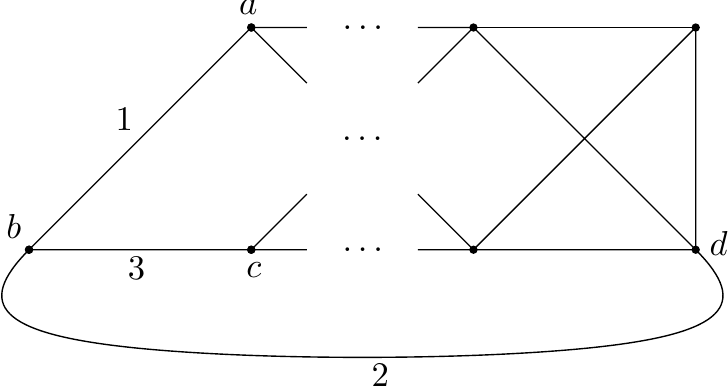}
		\caption{$G$}\label{fig anti X}
\end{figure}

When $|V(G)|=7$, this corresponds to $P_{6,3}$ in \cite{Schnetz2010}. \\

\begin{proof} We calculate 
	\[\Psi_2^{1,3}=\Phi^{\{a,c\},\{b\},\{d\}}\]
	\[\Psi^{12,32}=\Phi^{\{a,c,d\},\{b\}}\]
	When we remove edges $1,2,3$, vertex $b$ will be disconnected. On the graph $H$ in figure~\ref{fig anti H}, the polynomials are $\Psi_H$ and $\Phi_H^{\{a,b\},\{z\}}$. \\
	\begin{figure}[h]
		\includegraphics[scale=0.75]{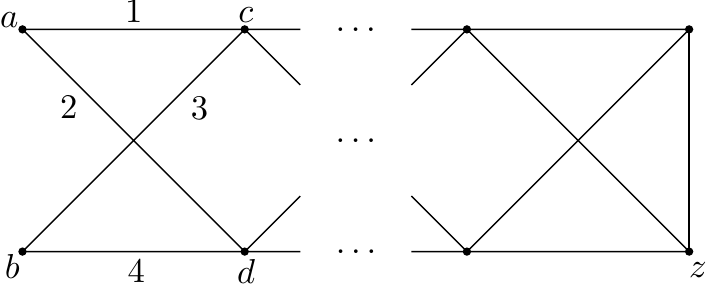}
		\caption{$H$}\label{fig anti H}
	\end{figure} 
	Notice that $H$ is invariant under swapping of vertices $c,d$.  As usual proceed by edge assignments according to lemma~\ref{lem c2 coeff} Any edge assignment that is not invariant under this swap is irrelevant modulo 2. We show that there are no such invariant assignments. \\
	If there were such an assignment, one factor would need edges $1,2$ or $3,4$. Excluding the assignment where $1,2,3,4$ appear in one term (which would create a cycle and so is invalid), these are the only possibilities. \\
	Of course, neither $1,2$ nor $3,4$ can be assigned to either $\Psi_H$ or $\Phi_H^{\{a,b\},\{z\}} $. \\ \\
	Therefore, $c_2^{(2)}(G)=0$ for $|V(G)|\geq7$.  
\end{proof}

\begin{proposition}
	Let $G$ be a decompleted symmetric X-ladder, with $|V(G)|\geq 7$, labelled as in figure \ref{fig sym X}. Then $c_2^{(2)}(G)=0$. \\
\end{proposition}
\begin{figure}[h]
	\includegraphics[scale=0.75]{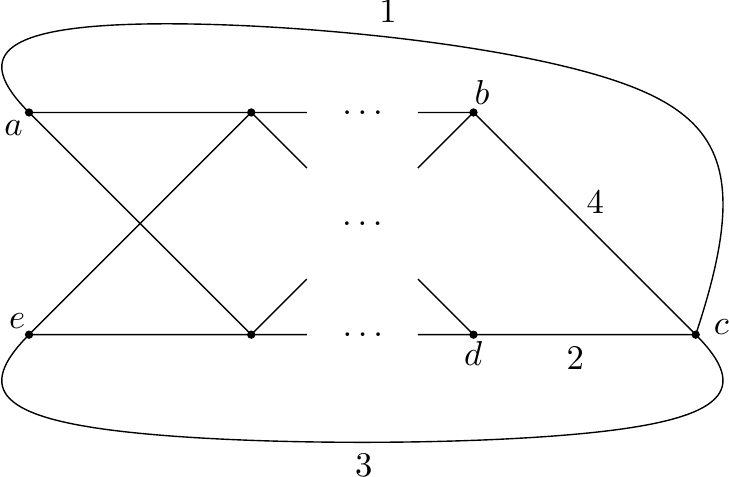}
	\caption{$G$}\label{fig sym X}
\end{figure}
\begin{proof}
We calculate 
\[\Psi^{12,34}=\pm\Phi^{\{a,b\},\{c\},\{d,e\}}\pm\Phi^{\{a,e\},\{c\},\{b,d\}}\]
\[\Psi^{13,24}=\pm\Phi^{\{a,b\},\{c\},\{d,e\}}\pm\Phi^{\{a,d\},\{c\},\{b,e\}}\]
We remove edges $1,2,3,4$ and the isolated vertex $c$. On the graph $H$ in figure \ref{fig sym H}, we get $\pm\Phi^{\{a,y\},\{b,z\}}\pm\Phi^{\{a,b\},\{y,z\}}$ and $\pm\Phi^{\{a,y\},\{b,z\}}\pm\Phi^{\{a,z\},\{b,y\}}$. \\
\begin{figure}[h]
	\includegraphics[scale=0.75]{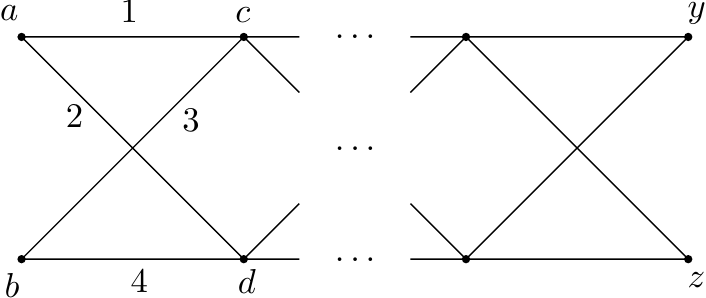}
	\caption{$H$}\label{fig sym H}
\end{figure}
As in the previous proposition, the graph $H$ is invariant under swapping of vertices $c$ and $d$. Therefore, for a non-zero contribution, we must be able to assign either edges $1,2$ or $3,4$ to one of the polynomials above. However, neither of these represent a valid assignment of edges to any of the above terms. \\ \\
Therefore, $c_2^{(2)}(G)=0$ for $|V(G)|\geq 7$. 
\end{proof}

\section{Tractability of these Methods}
Thus far, we have looked only at $c_2$ invariants when $p=2$. Indeed, using these methods at $p=3$ and above yields far too many cases than can be conveniently calculated without the aid of a computer. However, with such aid, the $c_2$ invariant at any fixed prime $p$ can be calculated for any sufficiently recursive family of graphs in a finite amount of time for all graphs of the family.  This generalizes results of \cite{Yeats2016} which had the finiteness result but only for certain families of circulants.  The results of \cite{Yeats2016} artificially and unnecessarily restricted the number of initial edges before the recursive structure begins and so applied to a vastly smaller class of graphs.

Note that the $c_2$ invariants calculated in the present paper were all $0$ so we always found complete cancellations.  The general picture is a little different.  Proposition~\ref{prop 3tor} gives the best illustration: the two cases each reduced the product of polynomials on $H_n$ to a similar product of polynomials on $H_{n-3}$.  In the case of proposition~\ref{prop 3tor} the two cases cancelled but in general they need not and so we would obtain a recurrence.  \cite{Yeats2016} gives other explicit examples where such recurrences are necessary, but the method holds much more generally than was appreciated therein.

We first take the notion of a recursively constructible family of graphs from \cite{NoyRibo2004}.  Intuitively a recursively constructible family of graphs is a family which is built from an initial graph by a repeated fixed sequence of certain basic graph operations.  Edge deletion is included in the allowable operations so in particular each element of the family can have edges connecting back to the initial piece.

We need to formalize this notion, see \cite{NoyRibo2004} section 2.  Given a graph $G$ and a set $U\subseteq V(G)$ let $N_G(U)$ be the neighbourhood of $U$ in $G$, that is the set of vertices of $G$ adjacent to some vertex in $U$.
\begin{definition}
  A sequence of graphs $\{G_n\}_{n\geq0}$ is a \textbf{recursively constructible family of graphs} if there exists a positive integer $r$ and a labelled graph $M$ such that
  \begin{itemize}
  \item $V(G_0) = W_0$, $E(G_0)=E_0$.
  \item $V(G_n) = V(G_{n-1})\cup W_n$.
  \item $N_{G_n}(W_n)\subseteq W_0 \cup \left(\bigcup_{i=0}^{r}W_{n-i}\right)$ for $n> r$.
  \item $E(G_n) = (E(G_{n-1})- S) \cup E_n$ where $S \subseteq \bigcup_{i=1}^rE_{n-i}$
  \item The graph induced by $W_0\cup\left(\bigcup_{i=0}^{r}W_{n-i}\right)$ in $G_n$ equals $M$ for $n>r$.
  \end{itemize}
\end{definition}
Note that the ``equals'' in the last point is not isomorphism; the labels must also match with the graph induced by $W_n$ in each $G_n$ always being the same as a labelled graph.

Noy and Rib\'{o} observe (\cite{NoyRibo2004} section 2) that these conditions imply that the operations used to move from $G_{n-1}$ to $G_n$ for $n>r$ in a recursive family can only be the following.
\begin{itemize}
\item Adding vertices ($W_n$) and edges ($E_n$) incident only to vertices in $W_0\cup\left(\bigcup_{i=0}^{r}W_{n-i}\right)$ in a way which is independent of $n$.
\item Removing edges with one end in $W_0$ and the other end in $\bigcup_{i=0}^{r}W_{n-i}$
\end{itemize}

%

Note that every family of graphs in this paper (either before or after decompletion) is a recursively constructible family of graphs.
Before proving our algorithmic theorem, we give the following lemma. 
\begin{lemma}\label{lem process edges}
	Given a spanning forest polynomial on a graph $G$, any assignment of edges yields some sum of spanning forest polynomials on the graph with those edges removed and any isolated vertices removed.  Furthermore, the vertices involved in the partitions defining the new spanning forest polynomials involve only vertices already in partitions for the input spanning forests and vertices incident to the assigned edges.
\end{lemma}
\begin{proof}
  We show that one edge deletion satisfies the lemma, as does one edge contraction. Then iterating the process, one can get any possible edge assignment.
  
  If some sequence of these contractions and deletions creates an impossible assignment (one that has cycles, connects vertices from different parts of the partition, etc.) then the corresponding collection of spanning forest polynomials is trivial --- it is equal to zero.
  
	Given some edge $e=\{u,v\}\in G$ and a spanning forest polynomial $\Phi_G^P$, a few cases arise.

        Cutting $e$ is most straightforward; this corresponds to not assigning $e$ to $\Phi_G^P$.  In this case we simply obtain $\Phi_{G\backslash e}^P$.  If this does not result in any isolated vertices then we are done.  If it does then we want to express the result as a spanning forest polynomial on the graph with that vertex $v$ removed.  If $v$ is in no part or in a part of size $>1$ then this cannot occur so we simply get $0$ which is a spanning forest polynomial.  If $v$ is a part by itself, then removing the vertex from the graph and the part from the partition gives the same polynomial now as a spanning forest polynomial on the desired graph.

        Now consider contracting $e$; this corresponds to assigning $e$ to $\Phi_G^P$.  If $u, v$ are in distinct parts of $P$ then we get $0$ which is allowed.

        Next suppose that $u,v$ are together in a part of $P$.  The resulting polynomial is $\Phi_{G/e}^{P'}$ where $P'$ is $P$ with $u,v$ identified.  To interpret this on $G\backslash e$, then, we break apart the tree corresponding to this part in such a way that $u$ and $v$ are in different halves of the tree.  Thus for each partition $P''$ resulting from further partitioning the part of $P$ containing $u$ and $v$ into two parts, one containing $u$ and the other containing $v$ we get the spanning forest polynomial for that partition.

        If one of $u$ or $v$ is in $P$ but the other is not, then we are in almost the same situation as the previous case except that we now need to add the other of $u$ or $v$ to the part containing the first and then further partition that part into two, one containing $u$ and the other containing $v$.

        Finally, suppose neither $u$ nor $v$ is in $P$.  Let $w$ be the vertex corresponding to $u$ and $v$ in $G/e$. Then similarly to the previous cases, this means that the tree containing $w$ in $G/e$ must be broken apart in $G\backslash e$ with $u$ and $v$ in different halves.  We don't know which part of $P$ this tree corresponds to, but if $P = P_1, P_2, \ldots , P_j$ then $\Phi_{G/e}^{P} = \sum_{i=1}^j \Phi_{G/e}^{P_1, \cdots, P_j\cup\{w\}, \cdots, P_j}$ and we can argue as in the previous case on each term of the sum.
%
\end{proof}

With the above lemma, we are able to prove the following result. 

\begin{theorem} Let $\{G_n\}_{n\geq 0}$ be a recursively constructible family of graphs with $2|V(G_n)| = |E(G_n)| + 2$ for $n$ sufficiently large.
The $c_2$ invariant for any fixed prime $p$ can be calculated using these methods in a finite amount of time for all graphs of the family.
  \label{thm finite}
\end{theorem}

Note that the condition $2|V(G_n)| = |E(G_n)| + 2$ is to guarantee the correct relationship between the degree and number of variables for using Lemma~\ref{lem c2 coeff} on the output of Proposition~\ref{prop calculate c2}.

\begin{proof}
Fix $p$.
  Let $G_m\in \{G_n\}_{n\geq 0}$, where $\{G_n\}_{n\geq 0}$ is a recursive family of graphs.  Let the $W_n$, $E_n$, and $S\subseteq \bigcup_{i=1}^r E_{n-i}$ be as in the definition of a recursively constructible family.  Let $H_n = G_n \backslash S$.

To begin with, assume that $|S|\geq 3$ and $m$ is sufficiently large that $2|V(G_m)| = |E(G_m)| + 2$ and $m>r$.

Starting with the Kirchhoff polynomial of $G_m$, by Proposition~\ref{prop calculate c2} and Lemma~\ref{lem c2 coeff} we can process between 3 and 5 edges of $S$ in order to calculate $c_2^{(p)}(G_m)$ by counting assignments of edges to certain products of $2(p-1)$ spanning forest polynomials.  Using Lemma~\ref{lem process edges} to assign the remaining edges of $S$ we can obtain an expression for $c_2^{(p)}(G_m)$ as a sum of edge assignments to products of spanning forest polynomials of $H_m$.  Furthermore, the vertices involved in the partitions can only be endpoints of edges in $\bigcup_{i=1}^r E_{m-i}$, that is they must be vertices of the copy of $M$ for $G_m$.

Next, using Lemma~\ref{lem process edges}, assign the edges of $E_m$ which are not already assigned.  The remaining graph is now $H_{m-1}$.  On the polynomial side, each summand from the sum of products of spanning forest polynomials on $H_m$ has itself become a sum of products of spanning forest polynomials on $H_{m-1}$ and in both cases the vertices involves in the partitions must be in the appropriate copy of $M$.  We can do the same for any product of $p-1$ spanning forest polynomials on $H_{m}$ which involve only vertices of $M$ regardless of whether it appeared in the expansion of $c_2^{(p)}(G_m)$.

$M$ is a finite graph and so has a finite number of vertices.  So there are only a finite number of partitions of subsets of these vertices.  So there are also only finitely many lists of $2(p-1)$ such partitions.  Call this set of products of partitions $\mathcal{C}$.  We can view any product of $2(p-1)$ spanning forest polynomials on $H_m$ involving only vertices of $M$ as being such a list of partitions.  Therefore the map described in the previous paragraph which takes a product of $2(p-1)$ spanning forest polynomials on $H_m$ to a sum of such products on $H_{m-1}$ is a map from $\mathcal{C}$ to itself.  This map is independent of $m$ for $m$ sufficiently large because of the recursive structure of the family.

For $c\in \mathcal{C}$, let $a_{m,c}$ correspond to the cardinality over $\mathbb{F}_p$ of the variety defined by the vanishing of the product of spanning forest polynomials corresponding to $c$ at the level of $H_m$.  Then the map described above gives a system of linear recurrences relating the $a_{m,c}$ with the $a_{m-1, c}$.  Such a system is always solvable by standard finite techniques.
The values for any finite number of small values of $m$ can be computed directly in a finite amount of time by counting edge assignments or by working directly with the variety.  Hence we can obtain the base cases for the recursion and deal with $m\leq r$ and any other $m$ which is too small.
Furthermore the particular linear combination of $a_{m,c}$ giving $c_2^{(p)}(G_m)$ is also independent of $m$ for $m$ sufficiently large because of the recursive structure of the family.

Therefore, the solution to the system of linear recurrences gives an expression for $c_2^{(p)}(G_m)$ for all $m$.

Now suppose $|S|<3$. For $m$ sufficiently large, we can assign the edges of $S$ and also $E_m,E_{m-1},\dots$ as needed to assign 3 to 5 edges. This case then follows from the above.

\end{proof}

\section{Conclusion}
Non-skew toroidal grids are a large family of graphs for which we now know $c_2^{(2)}=0$.  Previous families with $c_2=0$ were known by double triangle or by small edge or vertex cuts (see \cite{BrownSchnetz2012, BrownSchnetzYeats2014}), but none of these apply to the toroidal grids. Other than the non-skew toroidal grid $\gamma_3\times \gamma_3$, we do not know if non-skew toroidal grids have $c_2^{(p)}=0$ for primes $p>2$.  Either way would be interesting giving either a new family of weight drop graphs (see \cite{BrownYeats2011} for more on weight in this sense) or giving a family of graphs with $c_2^{(2)}=0$ for reasons other than weight drop.

The $X$-ladder result hints at how some structure of higher weight drops may be visible to these techniques as the larger the ladder the more independent symmetries forced $c_2^{(2)}$ to be $0$.

Finally, in view of theorem~\ref{thm finite} we have in-principal algorithms for calculating $c_2$ invariants for a much larger class of families of graphs than was known before.  Unfortunately these algorithms grow exponentially in every interesting parameter, so they are not practical unless further simplifications can be found.  Some small cases with $p=2$ or $p=3$ are probably tractable and would be a good testing ground for the possibility of finding simplifications.  Furthermore, these in-principle algorithms tell us something about the kinds of solutions which can appear -- they must come from solving systems of recurrences.  This strongly restricts the kinds of sequences which can appear and is in striking contrast to the sequences which can appear when the graph is fixed and $p$ varies, see \cite{BrownSchnetz2013}.

\end{document}